\documentclass[10pt]{amsart}

\usepackage{graphicx}
\usepackage{amsfonts,amsmath,latexsym,amssymb,amsthm,verbatim}
\usepackage{mathpazo}
\usepackage[mathpazo]{flexisym}
\usepackage{breqn}

\newcounter{theoremcounter}
\newcounter{lemmacounter}

\newcounter{dummycounter}

\newcounter{corcounter}

\newcounter{emptycounter}

\newtheorem{theorem}[theoremcounter]{Theorem}

\newtheorem{lemma}[lemmacounter]{Lemma}

\newtheorem{corollary}[corcounter]{Corollary}

\numberwithin{equation}{section}
\numberwithin{lemmacounter}{section}
\numberwithin{propcounter}{section}
\numberwithin{corcounter}{section}
\numberwithin{conjcounter}{section}
\numberwithin{theoremcounter}{section}
\numberwithin{probcounter}{section}

\newcounter{eqncounter}

\numberwithin{equation}{eqncounter}

\def\IR{\mathbf R}

\def\IZ{\mathbf Z}
\def\IN{\mathbf N}

\def\vNull{{\mbox{\boldmath $0$}}}

\def\vx{{\bf x}}
\def\vy{{\bf y}}

\def\Da{n}

\def\V{\textup{V}}
\def\Vol{\textup{Vol}}

\def\gi{g_i}
\def\Gt{G_\theta}
\def\Gi{G_i}

\def\GtGmN{\varphi_{-N+1}}
\def\GtGN{\varphi_N}
\def\GtGi{\varphi_i}
\def\ISone{\tau_1}
\def\IStwo{\tau_2}
\def\ISthree{\tau_3}
\def\ISfour{\tau_4}
\def\ISj{\tau_j}
\def\Lambdaj{\Lambda_j}

\def\F{\varepsilon}
\def\CL{C_L}
\def\M{M}
\def\dia{d}
\def\Cim{C_5}
\def\Cone{C_1}
\def\Ctwo{C_3}
\def\Cthree{C_4}
\def\Cfour{C_2}
\def\Rone{H_1}
\def\Rtwo{H_2}
\def\Rthree{H_3}
\def\Rfour{H_4}
\def\Rj{H_j}
\def\Zone{Z_1}
\def\Ztwo{Z_2}
\def\Zthree{Z_3}
\def\Zfour{Z_4}
\def\Zj{Z_j}
\def\Eone{R_1}
\def\Etwo{R_2}
\def\Ej{R_j}

\def\A{\Delta_x}
\def\B{\Delta_y}

\def\Dn{D_n}
\def\D3{D_3}
\def\logplus{\log^{+}}
\def\lphii{\iota_i}
\def\lphi1{\iota_1}
\def\lphiM{\iota_M}

\begin{document}
\title{Asymptotic Diophantine approximation: The multiplicative case}

\author{Martin Widmer}

\address{Department of Mathematics\\ 
Royal Holloway, University of London\\ 
TW20 0EX Egham\\ 
UK}

\email{martin.widmer@rhul.ac.uk}

\date{\today}

\subjclass[2010]{Primary 11J13; 11J25; 11J54 Secondary 11H46; 37A17}

\keywords{Diophantine approximation, reciprocals of fractional parts, flows, counting, Littlewood's conjecture}

\begin{abstract}
Let $\alpha$ and $\beta$ be irrational real numbers and $0<\F<1/30$. 
We prove a precise estimate for the number of positive integers $q\leq Q$ 
that satisfy $\|q\alpha\|\cdot\|q\beta\|<\F$.
If we choose $\F$ as a function of $Q$ we get asymptotics as $Q$ gets large,
provided $\F Q$ grows quickly enough in terms of the (multiplicative) 
Diophantine type of $(\alpha,\beta)$, e.g., if  $(\alpha,\beta)$ is a 
counterexample to Littlewood's conjecture then we only need that 
$\F Q$ tends to infinity. 
Our result yields a new upper bound on sums of reciprocals of products 
of fractional parts, and sheds some light on a recent question of L\^{e} 
and Vaaler.  
\end{abstract}

\begin{comment}
Abstract:
%Let $\alpha$ and $\beta$ be real numbers and suppose $q\|q\alpha\|\|q\beta\|\geq \phi(q)$
%for a positive monotone non-increasing function $\phi$.
We prove lower and upper bounds for the number of positive integers $q\leq Q$ 
that satisfy $\|q\alpha\|\cdot\|q\beta\|<\F$. If $\F=\F(Q)$ is chosen as a function of $Q$ and 
$\F(Q)Q$ grows quickly enough in terms of the (multiplicative)
Diophantine type of $(\alpha,\beta)$ then we get asymptpotics as $Q$ gets large. Our estimate yields a
new upper bound for the sum of reciprocals of products of fractional parts.\\
\end{comment}

\maketitle

\section{Introduction}\label{intro}
Let $\alpha$ and $\beta$ be irrational real numbers, and let $\|\cdot\|$ be the distance to the nearest integer. Littlewood's conjecture asserts that 
$$\liminf_{q\rightarrow \infty} q\cdot\|q\alpha\|\cdot\|q\beta\|=0.$$
We assume that $\phi:[1,\infty)\rightarrow (0,1/4]$ is a a non-increasing\footnote{By non-increasing we mean that $x,y \in [1,\infty)$ and $x\leq y$ implies that $\phi(x)\geq \phi(y)$.} function (depending on $(\alpha,\beta)$) such that 
\begin{alignat}1\label{multapproxcond}
q\cdot\|q\alpha\|\cdot\|q\beta\|\geq \phi(q)
\end{alignat}
for all positive integers  $q$. Note that $\phi$ can be chosen to be constant
if and only if the pair $(\alpha,\beta)$ is a counterexample to Littlewood's conjecture.
The condition (\ref{multapproxcond}) has been considered in various forms, e.g., by
Badziahin \cite{Badziahin2013}. He takes a function $f:\IN\rightarrow (0,\infty)$ and considers the set
\begin{alignat}1\label{Madf}
\text{\bf Mad}(f)=\biggl\{(\alpha,\beta)\in \IR^2; \liminf_{q\rightarrow \infty}f(q)\cdot q\cdot\|q\alpha\|\cdot\|q\beta\|>0\biggr\}.
\end{alignat}
Special cases of these sets
already appeared in \cite{BadziahinVelani2011}. If we assume that $1/f$ is also non-increasing then
$(\alpha,\beta)$ lies in $\text{\bf Mad}(f)$ if and only if (\ref{multapproxcond}) holds true
with a $\phi$ satisfying $1/f(q)\ll_{\alpha,\beta}\phi(q)\ll_{\alpha,\beta}1/f(q)$. \\

Throughout this article, let $Q\geq 1$,
$\F>0$, and $T>0$ be real numbers, and assume 
\begin{alignat}1\label{FTe}
\F/T^2\leq 1/e^2,
\end{alignat}
where $e$ denotes the base of the (natural) logarithm.
We consider the finite set
\begin{alignat*}5
M_{\alpha,\beta}(\F,T,Q)=\biggl\{(p_1,p_2,q)\in \IZ^3; &|p_1+q\alpha| \cdot|p_2+q\beta|<\F,\\
& \max\{|p_1+q\alpha|,|p_2+q\beta|\}\leq T,\\
&0<q\leq Q \biggr\}.
\end{alignat*}
\begin{theorem}\label{thmasympvol}
Suppose that (\ref{multapproxcond}) and (\ref{FTe}) hold, and set $\Cone=3^{28}$. Then we have
\begin{alignat*}3
\left||M_{\alpha,\beta}(\F,T,Q)|-4\F Q\left(\log\left(\frac{T^2}{\F}\right)+1\right)\right|
\leq \Cone(1+2T)^2\log\left(\frac{T^2}{\F}\right)\left(\frac{\F Q}{\phi(Q)}\right)^{2/3}.
\end{alignat*}
\end{theorem}
We shall see that the main term is just the volume of the set $Z$ defined in Section \ref{prereq}.
The constant $\Cone$ could easily be improved.
Choosing $T=1/2$ we have 
\begin{alignat*}3
|M_{\alpha,\beta}(\F,1/2,Q)|=|\{q\in \IZ; \|q\alpha\| \cdot\|q\beta\|<\F, 0<q\leq Q\}|
\end{alignat*}
which is of particular interest,
and hence we state this case of Theorem \ref{thmasympvol} as a corollary.
\begin{corollary}\label{corasympvol}
Suppose that (\ref{multapproxcond}) holds, that $0<\F\leq 1/(2e)^2$, and set $\Cfour=4\Cone=4\cdot3^{28}$. Then we have 
\begin{alignat*}3
\left||M_{\alpha,\beta}(\F,1/2,Q)|-4\F Q\left(1-\log(4\F)\right)\right|
\leq -\Cfour\log(\F)\left(\frac{\F Q}{\phi(Q)}\right)^{2/3}.
\end{alignat*}
\end{corollary}
If we choose a value $\F=\F(Q)\leq 1/(2e)^2$ for each value of $Q$, and we let $Q$ tend to infinity then we get 
asymptotics for $|M_{\alpha,\beta}(\F,1/2,Q)|$ provided $1/\phi(Q)=o(\sqrt{\F(Q)Q})$.
Let us write $\logplus Q=\max\{1,\log Q\}$.
A result of Gallagher \cite{Gallagher1962} implies that for $f(q)=(\logplus q)^{\lambda}$ the set $\text{\bf Mad}(f)$
has full Lebesgue measure if $\lambda>2$ and measure zero when $\lambda\leq 2$.
Hence, if $\F(Q)\gg (\logplus Q)^{2\lambda}/Q$ with $\lambda>2$ then the asymptotics 
are given by the main term in Corollary \ref{corasympvol} for almost\footnote{With respect to the Lebesgue measure.} every pair $(\alpha,\beta)\in \IR^2$. 
Bugeaud and Moshchevitin \cite{BugeaudMoshchevitin2011} showed that when $\lambda=2$ the set $\text{\bf Mad}(f)$ still has
full Hausdorff dimension. This was substantially improved by Badziahin \cite{Badziahin2013}
who showed that even with $f(q)=(\logplus q)(\logplus(\logplus q))$ the set $\text{\bf Mad}(f)$ has full Hausdorff dimension.\\

We now discuss an application of Corollary \ref{corasympvol}. In \cite{LeVaaler2015} L\^e and Vaaler showed that
\begin{alignat*}3
Q(\logplus Q)^2\ll \sum_{q=1}^{\lfloor Q\rfloor}(\|q\alpha\|\cdot\|q\beta\|)^{-1}.
\end{alignat*}
Motivated by this they raised the question whether there exist real irrational numbers $\alpha, \beta$ such that
\begin{alignat}3\label{LeVaalerub}
\sum_{q=1}^{\lfloor Q\rfloor}(\|q\alpha\|\cdot\|q\beta\|)^{-1}\ll_{\alpha,\beta} Q(\logplus Q)^2.
\end{alignat}
L\^e and Vaaler showed that (\ref{LeVaalerub}) holds for $(\alpha,\beta)$ provided the latter is a counterexample to Littlewood's conjecture,
i.e., provided one can choose $\phi$ from (\ref{multapproxcond}) to be a constant function. We show that $\phi(Q)\gg_{\alpha,\beta} 1/(\logplus Q)$ suffices.
\begin{corollary}\label{sum}
Suppose that (\ref{multapproxcond}) holds,  and set $\Ctwo=12$ and $\Cthree=3^{32}$. Then we have
\begin{alignat*}3
\sum_{q=1}^{\lfloor Q\rfloor}(\|q\alpha\|\cdot\|q\beta\|)^{-1}\leq \Ctwo Q\left(\log\left(\frac{Q}{\phi(Q)}\right)\right)^2+\Cthree\frac{Q}{\phi(Q)}\log\left(\frac{Q}{\phi(Q)}\right).
\end{alignat*}
\end{corollary}
Einsiedler, Katok and Lindenstrauss \cite{EinsiedlerKatokLindenstrauss} showed that the set of counterexamples to Littlewood's conjecture has Hausdorff dimension
zero, and it is widely believed that no such counterexample exists at all. On the other hand there is evidence for the 
existence of pairs $(\alpha,\beta)$ with $\phi(Q)\gg_{\alpha,\beta} 1/(\logplus Q)$. In fact, Badziahin and Velani \cite[(L2)]{BadziahinVelani2011} (see also \cite[Conjecture 1]{Badziahin2013})
conjectured that the set of these pairs has full Hausdorff dimension. Unfortunately, it is not known whether such a pair $(\alpha,\beta)$ really
exists and so we cannot unconditionally answer L\^e and Vaaler's question.

\begin{comment}
Recall that $\alpha$ is called badly approximable if there exists $C>0$ such that $q\|q\alpha\|>C$ for all positive integers $q$.
For $q>0$ and badly approximable $\alpha$ and $\beta$ 
we have  
$$\min\{|p_1+q\alpha|,|p_2+q\beta|\}>C_3/q,$$ 
and hence, $0<q<Q$ and
$|p_1+q\alpha| |p_2+q\beta|<{\psi(Q)}/{Q}$ implies $\max\{|p_1+q\alpha|,|p_2+q\beta|\}\leq \psi(Q)/C_3$. Thus, for
$T\geq \psi(Q)/C_3$ the second inequality in the definition of $M_{\alpha,\beta}(\psi,T,Q)$ is void. 
Applying Theorem \ref{thmasympvol} with $T=\psi(Q)/C_3$  and $Q_0=(eC_3)^2$
we obtain a third corollary. 
\begin{corollary}\label{corasympvol2}
Suppose $\alpha$ and $\beta$ are both badly approximable and such that (\ref{multapproxcond}) holds and $Q\geq(eC_3)^2$. Then, 
\begin{alignat*}3
\left||\{(p_1,p_2,q)\in \IZ^3; |p_1+q\alpha| |p_2+q\beta|<\frac{\psi(Q)}{Q}, 0<q<Q\}|-4\psi(Q)\log Q\right|\\
\leq (3C_0/C_3)\left(\log Q\left(\frac{\psi(Q)}{\phi(Q)}\right)^{2/3}+\psi(Q)\left(\frac{\psi(Q)}{\phi(Q)}\right)^{1/3}\right).
\end{alignat*}
\end{corollary}
\end{comment}

\section{Prerequisites}\label{prereq}

%We start with the simple observation that if $\F Q<\phi(Q)$ then Theorem \ref{thmasympvol} holds true.
\begin{lemma}\label{FQsmall}
If $\F Q<\phi(Q)$ then the stated inequality in Theorem \ref{thmasympvol} holds true.
\end{lemma}
\begin{proof}
Suppose $(p_1,p_2,q)\in M_{\alpha,\beta}(\F,T,Q)$. Hence, $1\leq q\leq Q$ and 
$\|q\alpha\|\cdot\|q\beta\|\leq|p_1+q\alpha|\cdot |p_2+q\beta|<\F$. On the other  hand by (\ref{multapproxcond}), 
and using the monotonicity of $\phi$, we have
$\|q\alpha\|\cdot\|q\beta\|\geq \phi(q)/q\geq \phi(Q)/Q$. Thus, if $\F Q<\phi(Q)$ then $|M_{\alpha,\beta}(\F,T,Q)|=0$.
It remains to show that the main term is covered by the error term. 
As $T^2/\F\geq e^2$ we have $\log(T^2/\F)+1<2\log(T^2/\F)$.
Using that $\F Q< \phi(Q)\leq 1/4$ we see that $4\F Q<(\Cone/2)(\F Q/\phi(Q))^{2/3}$.
This shows that the main term is bounded by the error term, and this proves the lemma.
\end{proof}
For the proof of Theorem  \ref{thmasympvol}  we thus can and will assume that
\begin{alignat}3\label{FQphi}
\F Q\geq \phi(Q).
\end{alignat}

For a vector $\vx$ in $\IR^\Da$ we write $|\vx|$ 
for the Euclidean length of $\vx$. 
Let $\Lambda$ be a lattice of rank $\Da$ in $\IR^\Da$. We define the first successive minimum 
$\lambda_1(\Lambda)$ of $\Lambda$ as the shortest Euclidean length of a non-zero lattice vector
\begin{alignat*}3
\lambda_1=\inf \{|\vx|; \vx\in \Lambda, \vx\neq \vNull\}.
\end{alignat*}
From now on suppose $\Da\geq 2$, $M\geq 1$ is also an integer, and let $L$ be a non-negative real.
We say that a set $S$ is in Lip$(\Da,M,L)$ if 
$S$ is a subset of $\IR^\Da$, and 
if there are $M$ maps 
$\lphi1,\ldots,\lphiM:[0,1]^{\Da-1}\longrightarrow \IR^\Da$
satisfying a Lipschitz condition
\begin{alignat*}3
|\lphii(\vx)-\lphii(\vy)|\leq L|\vx-\vy| \text{ for } \vx,\vy \in [0,1]^{\Da-1}, i=1,\ldots,M 
\end{alignat*}
such that $S$ is covered by the images
of the maps $\lphii$.\\

We will apply the following counting result which is an immediate consequence of \cite[Theorem 5.4]{art1}.
\begin{lemma}\label{MV_CL}
Let $\Lambda$ be a lattice of rank $\Da$ in $\IR^\Da$
with first successive minimum $\lambda_1$.
Let $S$ be a set in $\IR^\Da$ such that
the boundary $\partial S$ of $S$ is in Lip$(\Da,M,L)$.
Then $S$ is measurable, and moreover,
\begin{alignat*}3
\left||\Lambda\cap S|-\frac{\Vol S}{\det \Lambda}\right|
\leq \Dn \M\left(1+\left(\frac{L}{\lambda_1}\right)^{\Da-1}\right),
\end{alignat*}
where $\Dn=\Da^{2\Da^2}$.
\end{lemma}
Next we introduce the sets
\begin{alignat*}1
H&=\{(x,y)\in \IR^2; |xy|<\F, |x|\leq T, |y|\leq T\},\\
Z&=H\times (0,Q],
\end{alignat*}
and the lattice
\begin{alignat*}1
\Lambda&=(1,0,0)\IZ+(0,1,0)\IZ+(\alpha,\beta,1)\IZ.
\end{alignat*}
Clearly,
\begin{alignat}1\label{MLZ1}
|M_{\alpha,\beta}(\F,T,Q)|=|\Lambda\cap Z|.
\end{alignat}
Instead of working with $Z$ it is more convenient to decompose $Z$ into four identically shaped parts $\Zj$
and two rectangles $\Ej$. We set
\begin{alignat*}1
\Rone&=\{(x,y)\in \IR^2; |xy|< \F,0< x\leq T,0< y\leq T\},\\
\Rtwo&=\{(x,y)\in \IR^2; |xy|< \F,-T\leq x<0,0< y\leq T\},\\
\Rthree&=\{(x,y)\in \IR^2; |xy|< \F,0< x\leq T,-T\leq y<0\},\\
\Rfour&=\{(x,y)\in \IR^2; |xy|< \F,-T\leq x<0, -T\leq y<0\}.
\end{alignat*}
Furthermore, we put  for $1\leq j\leq 4$
\begin{alignat*}1
\Zj&=\Rj\times (0,Q],\\
\Eone&=[-T,T]\times \{0\} \times(0,Q],\\
\Etwo&=\{0\}\times[-T,T] \times(0,Q],
\end{alignat*}
so that we have the following partition
\begin{alignat*}1
Z=\Zone\cup \Ztwo\cup\Zthree\cup\Zfour\cup\Eone\cup \Etwo.
\end{alignat*}
Due to the irrationality of $\alpha$ and $\beta$ we have 
\begin{alignat*}1
|\Lambda\cap \Eone|=|\Lambda \cap\Etwo|=2\lfloor T\rfloor+1<2(T+1).
\end{alignat*}
Hence,
\begin{alignat*}1
\left||\Lambda\cap Z|-\sum_{j=1}^{4}|\Lambda\cap \Zj|\right|< 4(T+1).
\end{alignat*}
Using the automorphisms defined by $\ISone(x,y,z)=(x,y,z)$,
$\IStwo(x,y,z)=(-x,y,z)$, $\ISthree(x,y,z)=(x,-y,z)$, and $\ISfour(x,y,z)=(-x,-y,z)$ we have
$\ISj \Zj=\Zone$. Setting for $1\leq j\leq 4$
\begin{alignat*}1
\Lambdaj=\ISj(\Lambda),
\end{alignat*}
we find
\begin{alignat}1\label{MLZ2}
\left||M_{\alpha,\beta}(\F,T,Q)|-\sum_{j=1}^{4}|\Lambdaj\cap \Zone|\right|< 4(T+1).
\end{alignat}
Unfortunately, our set $\Zone$ is increasingly distorted when approaching the coordinate-axes.
After the trivial decomposition of $Z$ we shall now consider a less obvious decomposition of
our new counting domain $\Zone$.
  
\section{Partitioning the counting domain}\label{part}
 
First let us decompose $\Rone$ into three disjoint pieces. Set 
\begin{alignat*}1
\A&=\{(x,y);0<y<(\F/T^2)x, 0<x\leq T\},\\
\B&=\{(x,y);(T^2/\F)x\leq y\leq T, 0<x<\F/T\},\\
S&=\{(x,y);0<(\F/T^2)x\leq y< (T^2/\F)x, xy<\F\}.
\end{alignat*}
Hence, we have
\begin{alignat}1\label{Zdecomp}
|\Lambdaj\cap \Zone|=|\Lambdaj\cap \A\times(0,Q]|+|\Lambdaj\cap \B\times(0,Q]|+|\Lambdaj\cap S\times(0,Q]|.
\end{alignat}
The sets $\A$ and $\B$ are long and thin triangles, distorted only in $x$-direction or $y$-direction respectively.
The set $S$ is more troublesome and requires a further decomposition into about $-\log(\F/T^2)$ pieces.
Recall that by hypothesis $0<\F/T^2\leq 1/e^2$.
Let $\nu\in [1/e^2,1/e]$ be maximal such that $N=\log(\F/T^2)/\log \nu$ is an integer.
Hence,
\begin{alignat}1\label{NQT}
1\leq N\leq -\log(\F/T^2).
\end{alignat}
Decompose $S$ into the $2N$ pieces $S_{-N+1},\ldots, S_{N}$, where
$$S_i=\{(x,y); 0<\nu^{i}x\leq y<\nu^{i-1}x, xy<\F\}.$$
Then we have the following partition
\begin{alignat}1\label{Sdecomp}
S=\bigcup_{-N+1\leq i\leq N}S_i.
\end{alignat}
Note that 
\begin{alignat}1\label{S0bound}
S_0\subset [0,\sqrt{\F}]\times [0,\sqrt{\F/\nu}]\subset[0,3\sqrt{\F}]^2. 
\end{alignat}
A straightforward calculation yields
\begin{alignat*}1
\Vol_2(S_0)=\frac{\sqrt{\F\nu}\sqrt{\F/\nu}}{2}+\int_{\sqrt{\nu\F}}^{\sqrt{\F}}\frac{\F}{x}dx -\frac{\F}{2}=-\frac{\F}{2}\log\nu.
\end{alignat*}
Hence,
\begin{alignat*}1
V=\Vol_3 (S_0\times (0,Q])=-\frac{\F}{2} Q\log\nu.
\end{alignat*}
Thus
\begin{alignat}1\label{Vpsi}
\frac{\F Q}{2}\leq V\leq \F Q.
\end{alignat}

\section{Applying flows}\label{flows}
In this section we construct certain elements of the diagonal flow on $\IR^3$
that transform our distorted sets into sets of small diameter. 
 
We introduce the following automorphisms of $\IR^2$
$$\gi(x,y)=(\nu^{i/2}x,\nu^{-i/2}y).$$
Then we have for $-N+1\leq i\leq N$
$$\gi S_i=S_0.$$
We extend $\gi$ to an automorphism of $\IR^3$ 
$$\Gi(x,y,z)=(\nu^{i/2}x,\nu^{-i/2}y,z),$$ so that
$$\Gi(S_i\times(0,Q])=S_0\times(0,Q].$$
Next we introduce a further automorphism of $\IR^3$ 
$$\Gt(x,y,z)=(\theta x,\theta y,\theta^{-2} z),$$
where    
$$\theta=\frac{V^{1/3}}{\sqrt{\F}}.$$
Let us write 
$$\GtGi=\Gt\circ\Gi.$$
Then we have 
\begin{alignat*}1
\GtGi(S_i\times(0,Q])=\theta S_0\times (0,\theta^{-2}Q].
\end{alignat*}
Combining (\ref{S0bound}) and (\ref{Vpsi}) we get
\begin{alignat}1\label{Superset1}
\GtGi(S_i\times(0,Q])=\theta S_0\times (0,\theta^{-2}Q]\subset [0,3\theta\sqrt{\F}]^2\times (0,\theta^{-2}Q]\subset [0,3V^{1/3}]^3.
\end{alignat}
Similarly, we find
\begin{alignat}1\label{Superset2}
\GtGN (\A\times (0,Q])&\subset [0,3V^{1/3}]^3,\\
\label{Superset3}\GtGmN (\B\times (0,Q])&\subset [0,3V^{1/3}]^3.
\end{alignat}

\begin{lemma}\label{Lip}
For $-N+1\leq i\leq N$ the boundary of $\GtGi (S_i\times (0,Q])$, $\GtGN (\A\times (0,Q])$ and $\GtGmN (\B\times (0,Q])$ lies in Lip$(3,M,L)$ 
where $\M=5$, $L=\CL V^{1/3}$ and $\CL=12$.
\end{lemma}
\begin{proof}
The boundary of the set $\GtGi (S_i\times (0,Q])=\theta S_0\times (0,\theta^{-2}Q]$ can be covered by $4$ planes and 
the set 
$$\{(x,\theta^2\F/x,z); \theta\sqrt{\nu \F}\leq x\leq \theta\sqrt{\F}, 0\leq z\leq Q\theta^{-2}\}.$$
For the Jacobian $J$ of the parameterising map 
$$(t_1,t_2)\rightarrow (at_1+b,\frac{\theta^2\F}{(at_1+b)},ct_2)$$
with $a=\theta\sqrt{\F}(1-\sqrt{\nu})$, $b=\theta\sqrt{\nu\F}$, $c=Q/\theta^2$, and domain $[0,1]^2$ 
we get for its $l_2$-operator norm $\|J\|_2\leq 4V^{1/3}$ which yields the required Lipschitz condition thanks to
the Mean-Value Theorem.
Hence, we are left with the linear pieces of the boundary. Clearly, a subset of a plane with diameter 
no larger than $\dia$ can be parameterised by a single affine map with domain $[0,1]^2$ and Lipschitz constant $2\dia$.  
Thus, it suffices to show that the diameter of $\GtGi (S_i\times (0,Q])$ is $\leq 6 V^{1/3}$. But the latter
holds due to (\ref{Superset1}).
 
Finally, the boundary of the set $\GtGN (\A\times (0,Q])$ and of the set $\GtGmN (\B\times (0,Q])$
can each be covered by $5$ planes. Moreover, by (\ref{Superset2}) and (\ref{Superset3}) their diameter is also $\leq 6 V^{1/3}$.
This proves the lemma.
\end{proof}

\section{Controlling the orbits}\label{orbits}
Our transformations of the previous section have brought our distorted sets into nice shapes. 
Unfortunately, they transform our lattices $\Lambdaj$ in a less favourable manner.
Indeed, the corresponding orbit of $\Lambdaj$ escapes to infinity, i.e., the fist successive minimum
gets arbitrarily small. However, the rate of escape is controllable and sufficiently slow.

\begin{lemma}\label{minbound}
For $1\leq j\leq 4$, $-N+1\leq i\leq N$, and $Q\geq 1$ we have 
\begin{alignat*}1
\lambda_1(\GtGi \Lambdaj)&\geq \min\{1,1/(2T)\}\phi(Q)^{1/3}.
\end{alignat*}
\end{lemma}
\begin{proof}
Let $v\in \Lambdaj$ be an arbitrary non-zero lattice point. Then there exist $\epsilon_1$ and $\epsilon_2$
in $\{-1,1\}$ and $p_1,p_2,q\in \IZ$, not all zero, such that $v=(\epsilon_1(p_1+q\alpha),\epsilon_2(p_2+q\beta),q)$. 
First suppose $q\neq 0$. Then by the inequality of arithmetic and geometric means we have 
\begin{alignat*}1
|\GtGi v|^2\geq 3 (|p_1+q\alpha|\cdot|p_2+q\beta|\cdot|q|)^{2/3}.
\end{alignat*}
Using our hypothesis (\ref{multapproxcond}) we get $|p_1+q\alpha|\cdot|p_2+q\beta|\cdot|q|\geq \phi(|q|)$. 
If $|q|\leq Q$ we conclude, by the monotonicity of $\phi$, that $\phi(|q|)\geq \phi(Q)$, and hence
$$|\GtGi v|\geq \phi(Q)^{1/3}.$$
If, on the other hand, $|q|> Q$ then, looking only at the last coordinate, and using (\ref{FQphi}), we find
$$|\GtGi v|> \theta^{-2}Q\geq (\F Q)^{1/3}\geq \phi(Q)^{1/3}.$$ 
Suppose now that $q=0$. Then $p_1$ and $p_2$ are not both zero, and hence
\begin{alignat*}1
|\GtGi v|\geq \max\{\theta\nu^{i/2}|p_1|,\theta\nu^{-i/2}|p_2|\}\geq \theta\nu^{|i|/2}\geq \theta\nu^{N/2}.
\end{alignat*}
Recall that $\nu^{N/2}=\sqrt{\F}/T$, and $\theta=\V^{1/3}/\sqrt{\F}\geq 2^{-1/3}\left(Q/\sqrt{\F}\right)^{1/3}$.
Hence,
\begin{alignat*}1
\theta\nu^{N/2}\geq \frac{1}{2T}\left(\F Q\right)^{1/3}\geq \frac{1}{2T} \phi(Q)^{1/3}.
\end{alignat*}
This proves the lemma.
\end{proof}

\section{Proof of Theorem \ref{thmasympvol}}\label{proof}
Let $1\leq j\leq 4$. Decomposing the set $\Zone$ using (\ref{Zdecomp}) and (\ref{Sdecomp}) and then applying the automorphisms $\varphi_i$ yields 
\begin{alignat*}3
|\Lambdaj\cap \Zone|=& &&|\Lambdaj\cap \A\times(0,Q]|+|\Lambdaj\cap \B\times(0,Q]|
+\sum_{i=-N+1}^N|\Lambdaj\cap S_i\times(0,Q]|\\
=& &&|\GtGN\Lambdaj\cap \GtGN(\A\times(0,Q])|\\
&+&&|\GtGmN\Lambdaj\cap \GtGmN(\B\times(0,Q])|\\
&+&&\sum_{i=-N+1}^N|\GtGi\Lambdaj\cap \GtGi(S_i\times(0,Q])|.
\end{alignat*}
Note that $\det\GtGi\Lambdaj=1$. Applying Lemma \ref{MV_CL} to each summand, using Lemma \ref{Lip},  and collecting the main terms and the error terms terms yields
\begin{alignat}3\label{estimate1}
\left||\Lambdaj\cap \Zone|-\Vol_3(\Zone)\right|
\leq2\D3\M\CL^2\sum_{i=-N+1}^N\left(1+\frac{\V^{2/3}}{\lambda_1(\GtGi \Lambdaj)^2}\right).
\end{alignat}
Then, applying Lemma \ref{minbound}, we see that the right hand-side of (\ref{estimate1}) is bounded by
\begin{alignat}3
\nonumber\leq 4\D3\M\CL^2\max\{1,2T\}^2N\left(1+\frac{\V^{2/3}}{\phi(Q)^{2/3}}\right). 
\end{alignat}
Using  that by (\ref{Vpsi}) $\V\leq \F Q$, and then again that $\F Q\geq \phi(Q)$ we conclude that the latter is bounded by
\begin{alignat}3
\nonumber&\leq 8\D3\M\CL^2(1+2T)^2N\left(\frac{\F Q}{\phi(Q)}\right)^{2/3}. 
\end{alignat}
Putting $\Cim=8\D3\M\CL^2=8\cdot 3^{18}\cdot 5\cdot 12^2$, and recalling that $N\leq \log(T^2/\F)$ we conclude that 
\begin{alignat*}1
\left||\Lambdaj\cap \Zone|-\Vol_3(\Zone)\right|&\leq \Cim(1+2T)^2\log\left(\frac{T^2}{\F}\right)\left(\frac{\F Q}{\phi(Q)}\right)^{2/3}.
\end{alignat*}
By virtue of inequality (\ref{MLZ2}), we get
\begin{alignat*}3
\left||M_{\alpha,\beta}(\F,T,Q)|-4\Vol_3(\Zone)\right|\leq &5\Cim(1+2T)^2\log\left(\frac{T^2}{\F}\right)\left(\frac{\F Q}{\phi(Q)}\right)^{2/3}.
\end{alignat*}
Finally, we note that $5\Cim<3^{28}=\Cone$ and
$$\Vol_3(\Zone)=\F Q\left(\log\left(\frac{T^2}{\F}\right)+1\right),$$
and this completes the proof of Theorem \ref{thmasympvol}.

\section{Proof of Corollary \ref{sum}}
We have
\begin{alignat*}3
\sum_{q=1}^{\lfloor Q\rfloor}(\|q\alpha\|\cdot\|q\beta\|)^{-1}&\leq \sum_{k=1}^{\infty} 2^{k+1}|\{q; 1\leq q\leq Q, 2^{-k-1}\leq \|q\alpha\|\cdot\|q\beta\|<2^{-k}\}|\\
&\leq \sum_{k=1}^{\infty} 2^{k+1}|\{q; 1\leq q\leq Q,  \|q\alpha\|\cdot\|q\beta\|<2^{-k}\}|\\
&= \sum_{k=1}^{\infty} 2^{k+1}|M_{\alpha,\beta}(2^{-k},1/2,Q)|.
\end{alignat*}
Moreover, in the proof of Lemma \ref{FQsmall} we have seen that $M_{\alpha,\beta}(\F,T,Q)=\emptyset$
when $\F<\phi(Q)/Q$. We apply this with $\F=2^{-k}$ and $T=1/2$.
%Suppose $1\leq q\leq Q$. As we have $q\|q\alpha\|\|q\beta\|\geq \phi(q)$ for all positive integers $q$, and since $\phi$ is monotone and non-increasing, we %conclude that $\|q\alpha\|\|q\beta\|\geq \phi(Q)/Q$. Hence,
Hence,
\begin{alignat}3
\nonumber\sum_{k=1}^{\infty} 2^{k+1}|M_{\alpha,\beta}(2^{-k},1/2,Q)|&= \sum_{k=1}^{\lfloor\log_2(Q/\phi(Q))\rfloor} 2^{k+1}|M_{\alpha,\beta}(2^{-k},1/2,Q)|\\
\label{estapplied}&<4\cdot2^5Q+ \sum_{k=5}^{\lfloor\log_2(Q/\phi(Q))\rfloor} 2^{k+1}|M_{\alpha,\beta}(2^{-k},1/2,Q)|.
\end{alignat}
From Corollary \ref{corasympvol} we get for integers $k\geq 5$
\begin{alignat}3\label{corollaryapplied}
%&|\{q\in \IZ; ||q\alpha|| ||q\beta||<2^{-k}, 0<q<Q\}|\\
|M_{\alpha,\beta}(2^{-k},1/2,Q)|\leq &4(\log2)Qk2^{-k}+\Cfour(\log2)k\left(\frac{2^{-k}Q}{\phi(Q)}\right)^{2/3}.
\end{alignat}
Combining (\ref{estapplied}) and (\ref{corollaryapplied}) Corollary \ref{sum} follows from a straightforward  calculation
using the trivial estimates  $\sum_{k=1}^{K}k\leq K^2$ and $\sum_{k=1}^{K}kx^k\leq Kx^{K+1}/(x-1)$ (where $x>1$) and that
$\phi(Q)\leq 1/4$.

\section*{acknowledgements}
This article was initiated during a visit at the University of York, and, in parts, motivated by a question of Sanju Velani.
I am very grateful to Victor Beresnevich, Alan Haynes and Sanju Velani for many interesting and stimulating discussions
and their encouragement. I also would like to thank Th\'ai Ho\`ang L\^{e} for fruitful discussions, drawing my attention to \cite{BadziahinVelani2011}, and for pointing out an error  in an early draft of the manuscript.\\

\bibliographystyle{amsplain}
\bibliography{literature}

\providecommand{\bysame}{\leavevmode\hbox to3em{\hrulefill}\thinspace}
\providecommand{\MR}{\relax\ifhmode\unskip\space\fi MR }
\providecommand{\MRhref}[2]{%
  \href{http://www.ams.org/mathscinet-getitem?mr=#1}{#2}
}
\providecommand{\href}[2]{#2}
\begin{thebibliography}{1}

\bibitem{Badziahin2013}
D.~Badziahin, \emph{On multiplicatively badly approximable numbers},
  {Mathematika} \textbf{59, no.1} (2013), 31--55.

\bibitem{BadziahinVelani2011}
D.~A. Badziahin and S.~Velani, \emph{Multiplicatively badly approximable
  numbers and generalised cantor sets}, {Adv. Math.} \textbf{228, no.5} (2011),
  2766--2796.

\bibitem{BugeaudMoshchevitin2011}
Y.~Bugeaud and N.~Moshchevitin, \emph{Badly approximable numbers and
  {L}ittlewood-type problems}, {Math. Proc. Cambridge Phil. Soc.} \textbf{150}
  (2011), 215--226.

\bibitem{EinsiedlerKatokLindenstrauss}
M.~Einsiedler, A.~Katok, and E.~Lindenstrauss, \emph{Invariant measures and the
  set of exceptions to {L}ittlewood's conjecture}, {Ann. of Math. (2)}
  \textbf{164} (2006), 513--560.

\bibitem{Gallagher1962}
P.~Gallagher, \emph{Metric simultaneous diophantine aproximations}, {J. London
  Math. Soc.} \textbf{37} (1962), 387--390.

\bibitem{LeVaaler2015}
T.~H. L\^{e} and J.~D. Vaaler, \emph{Sums of products of fractional parts}, {to
  appear in Proc. London Math. Soc.}

\bibitem{art1}
M.~Widmer, \emph{Counting primitive points of bounded height}, Trans. Amer.
  Math. Soc. \textbf{362} (2010), 4793--4829.

\end{thebibliography}

\end{document}